\documentclass[11pt,a4paper]{amsart}

\usepackage{amstext,amssymb,amsopn,amsthm,mathrsfs,amsfonts,amsmath,amsxtra}

\usepackage{dsfont,multirow,comment,enumerate,bm,float}

\usepackage{graphicx,color,xcolor}

\usepackage{orcidlink}

\allowdisplaybreaks

\newtheorem{theorem}{Theorem}[section]

\newtheorem{proposition}[theorem]{Proposition}

\newtheorem{lemma}[theorem]{Lemma}
\newtheorem{remark}[theorem]{Remark}

\newtheorem{bigthm}{Theorem}

\def\a{\alpha}
\def\b{\beta}

\allowdisplaybreaks

\title[Jacobi heat kernel]
	{New sharp bounds for the Jacobi heat kernel via \\ an extension of the Dijksma-Koornwinder formula}

\author[A.\ Nowak]{Adam Nowak \orcidlink{0000-0003-0814-6795}}
\address[Adam Nowak]{Institute of Mathematics,
Polish Academy of Sciences,
\'Sniadeckich 8,
00-656 Warszawa, Poland}
\email{anowak@impan.pl}

\author[P. Sj\"ogren]{Peter Sj\"ogren \orcidlink{0000-0001-7862-8584}}
\address[Peter Sj\"ogren]{
Mathematical Sciences, University of Gothenburg
and
Mathematical Sciences, Chalmers University of Technology,
SE-412 96 G\"oteborg, Sweden
}
\email{peters@chalmers.se}

\author[T.Z.\ Szarek]{Tomasz Z.\ Szarek \orcidlink{0000-0003-0821-5607}}
\address[Tomasz Z. Szarek]{
Department of Mathematics, University of Georgia,
Athens, GA 30602, USA
and
Instytut Matematyczny,
Uniwersytet Wroc{\l}awski,
Plac Grunwaldzki 2,
50-384 Wroc{\l}aw,
Poland}
\email{tzs10705@uga.edu}

\begin{document}


\begin{abstract}
We obtain sharp estimates for the Jacobi heat kernel in a range of parameters where
the result has not been established before. This extends and completes an earlier result due to the authors.
The proof is based on a generalization of the Dijksma-Koornwinder product formula for Jacobi polynomials.
\end{abstract}

\maketitle
\thispagestyle{empty}


\section{Introduction}

Let $\a,\b > -1$. The Jacobi heat kernel is defined by
\begin{equation} \label{ker_Jac}
G_t^{\a,\b}(x,y) := \sum_{n=0}^{\infty} e^{-t n(n+\a+\b+1)} \frac{P_n^{\a,\b}(x) P_n^{\a,\b}(y)}{h_n^{\a,\b}},
	\qquad x,y \in [-1,1], \quad t > 0,
\end{equation}
where $P_n^{\a,\b}$ are the classical Jacobi polynomials and $h_n^{\a,\b}$ are suitable normalizing
constants. This is the kernel of the Jacobi semigroup generated by the Jacobi operator; see Section \ref{sec:pre}
for precise definitions. The oscillatory series displayed above cannot be computed.

The kernel $G_t^{\a,\b}(x,y)$ is an important object in analysis and beyond,
in particular due to its deep connections with, in a sense, more elementary
heat kernels in numerous classical frameworks. These include, among others, the contexts of Euclidean spheres (cf.\ \cite{NSS1}),
compact rank-one symmetric spaces (cf.\ \cite{NSS2}), Fourier-Bessel expansions (cf.\ \cite{NoRo2})
and algebraic contexts related to multi-dimensional
balls, simplices, cones, conical surfaces and other domains, see e.g.\ \cite{NSS2,HK,DuXu,DaiXu} and references therein.
Thus advancing knowledge about $G_t^{\a,\b}(x,y)$ is widely motivated.

In this paper we complete our study \cite{NSS2} of the behavior of the Jacobi heat kernel by proving genuinely sharp two-sided
global estimates in the whole range of the parameters $\a$ and $\b$ indicated above.
In order to state the main result, we define for $t > 0$ and $\theta,\varphi \in [0,\pi]$ the quantity
$$
Z_t^{\a,\b}(\theta,\varphi) = \big[ t + \theta \varphi \big]^{-\a-1/2} \,
	\big[ t + (\pi-\theta)(\pi-\varphi)\big]^{-\b-1/2}\; \frac{1}{\sqrt{t}} \exp\bigg({- \frac{(\theta-\varphi)^2}{4t}}\bigg).
$$
We will prove the following theorem, which concludes
the long-standing open problem of finding sharp bounds for the Jacobi heat kernel.
\begin{bigthm} \label{thm:main}
Let $\a,\b > -1$ and $T > 0$ be fixed. There exists a constant $C=C(\a,\b,T) > 1$ such that
$$
C^{-1} Z_t^{\a,\b}(\theta,\varphi) \le G_t^{\a,\b}(\cos\theta,\cos\varphi) \le C Z_t^{\a,\b}(\theta,\varphi)
$$
for $\theta,\varphi \in [0,\pi]$ and $0 < t \le T$.
\end{bigthm}

In the restricted range of parameters $\a,\b \ge -1/2$, Theorem \ref{thm:main} was proved by the authors in \cite{NSS1,NSS2},
see \cite[Theorem 3.1]{NSS2}. We remark that only qualitatively sharp, short-time bounds for $G_t^{\a,\b}(x,y)$
were earlier obtained in \cite{CKP} and independently by other methods in \cite{NoSj};
in the latter paper under the restriction $\a,\b \ge -1/2$.
The difference between genuinely sharp and qualitatively sharp bounds is that in the latter case the constants appearing
in the exponential factors in the lower and upper estimates (which are $1/4$ in our Theorem \ref{thm:main})
differ from each other and from the optimal one.
This exponential gap swallows various polynomial factors, which in case of genuinely sharp bounds must be determined precisely.

On the other hand, for any $\a,\b > -1$ showing the long-time sharp estimates
\begin{equation} \label{1.XX}
C^{-1} \le G_t^{\a,\b}(x,y) \le C, \qquad x,y \in [-1,1], \quad t \ge T,
\end{equation}
with $T > 0$ fixed and some $C=C(\a,\b,T)>1$ independent of $x,y,t$, is much more straightforward, see \cite{CKP,NoSj}.

The new contribution of Theorem \ref{thm:main} pertains essentially (see \cite[Remark 3.2]{NSS2})
to the range of parameters when $\a < -1/2$ or $\b < -1/2$, in particular to the tricky case when also $\a+\b \le -3/2$.
These cases are more difficult to handle than the previously solved case $\a,\b \ge -1/2$.
In fact, some crucial tools and techniques elaborated and used earlier in \cite{NoSj,NSS2} break down for the parameters below
$-1/2$ and hence the proof of Theorem \ref{thm:main} requires some new ideas and tools. One of them is an extended variant
of a \emph{product formula} for Jacobi polynomials due to Dijksma and Koornwinder \cite{DK}. This extension, obtained
by means of analytic continuation,
is contained in Theorem \ref{thm:DK_ext}, see also Remark \ref{rem:complex}; we believe that this result is of independent interest.
For the statement of the original Dijksma-Koornwinder formula and further details, see Section \ref{sec:DK}.
The extended product formula leads to an extension of a \emph{reduction formula} for the Jacobi heat kernel
from \cite[Theorem 3.1]{NoSj}; see Theorem \ref{thm:red_ext}, which is the key to proving the desired bounds.

The paper is organized as follows.
In Section \ref{sec:pre} we introduce basic notation and gather various facts and results needed in the sequel.
Section \ref{sec:DK} is devoted to the extension of the Dijksma-Koornwinder product formula.
In Section \ref{sec:red} we extend the reduction formula for the Jacobi heat kernel.
Section \ref{sec:bnd} contains some preparations for the proof of Theorem \ref{thm:main}, given in Section \ref{sec:proof}.

\section{Preliminaries} \label{sec:pre}

Throughout the paper we use a standard notation.
We will frequently write $X \lesssim Y$ to indicate that $X \le C Y$ with a positive constant $C$
independent of significant quantities. We shall write $X \simeq Y$ when simultaneously $X \lesssim Y$ and $Y \lesssim X$.

\subsection*{Varia}
For various facts below, see e.g.\ \cite{Sz}. We use the following notation.
\begin{align*}
P_n^{\a,\b} & \qquad \textrm{Jacobi polynomials} \\
C_n^{\lambda}\;\; & \qquad \textrm{Gegenbauer/ultraspherical polynomials} \\
P_n^{\lambda}\;\; & \qquad \textrm{Jacobi polynomial} \;\; P_n^{\lambda,\lambda}.
\end{align*}
The Gegenbauer and Jacobi polynomials are linked by
$$
C_n^{\lambda}(x) = \frac{\Gamma(\lambda+1/2)}{\Gamma(2\lambda)}\,
\frac{\Gamma(n+2\lambda)}{\Gamma(n+\lambda+1/2)} \,P_n^{\lambda-1/2}(x),
	\qquad \lambda > -1/2, \quad \lambda \neq 0.
$$
We will use a similar shortened notation for the Jacobi heat kernel, $G_t^{\lambda}(x,y):=G_t^{\lambda,\lambda}(x,y)$,
and for other objects defined below, e.g.\ $h_n^{\lambda} := h_n^{\lambda,\lambda}$, $J^{\lambda} := J^{\lambda, \lambda}$.

For $\a,\b > -1$ the Jacobi orthogonality measure on the interval $[-1,1]$ is
$$
d\varrho_{\a,\b}(x) = (1-x)^{\a}(1+x)^{\b}\, dx.
$$
For complex $\a$ and $\b$ such that $\Re \a, \Re \b > -1$, we let
$$
h_n^{\a,\b}  =
\frac{2^{\a+\b+1} \Gamma(n+\a+1)\Gamma(n+\b+1)}{(2n+\a+\b+1)\Gamma(n+\a+\b+1)\Gamma(n+1)}, \qquad n \ge 0,
$$
where in the case $n=0$ the product $(2n+\a+\b+1)\Gamma(n+\a+\b+1)$ is replaced by $\Gamma(\a+\b+2)$.
For real $\a,\b > -1$, one has $h_n^{\a,\b} = \|P_n^{\a,\b}\|^2_{L^2(d\varrho_{\a,\b})}$.

Note that $P_0^{\a,\b} \equiv 1$.
We state some other useful formulas, valid for $n \ge 0$:
\begin{align}
P_n^{\a,\b}(1) & 
= \frac{(\a+1)_n}{n!} = \frac{\Gamma(n+\a+1)}{\Gamma(n+1)\Gamma(\a+1)}, \label{v1} \\
\frac{d}{dx} P_n^{\a,\b}(x) & = \frac{1}2 (n+\a+\b+1) P^{\a+1,\b+1}_{n-1} (x). \label{diff_Jac}
\end{align}
By convention, $P_{k}^{\a,\b} \equiv 0$ for $k < 0$.
The Pochhammer symbol is defined as $(a)_k = a(a+1)\cdot\ldots\cdot (a+k-1)$ for $k \ge 1$, with the convention that $(a)_0=1$.

\subsection*{Jacobi polynomials and complex values of the parameters}
In this subsection we consider $\a,\b \in \mathbb{C}$, see \cite[Chapter IV, Section 4.22]{Sz}.

Jacobi polynomials have natural extensions to all complex
parameters, since, defined a priori as orthogonal polynomials for $\a,\b>-1$,
they are polynomials not only in their argument, but also in the parameters.
One has $P^{\a,\b}_0 \equiv 1$, and for $n \ge 1$
\begin{align*}
P_n^{\a,\b}(x) & = \frac{1}{n!}\sum_{k=0}^n \binom{n}{k} (n+\a+\b+1)_k\; (\a+k+1)_{n-k}\; \Big(\frac{x-1}2\Big)^k.
\end{align*}
Note that $P_n^{\a,\b}$ is a polynomial of degree at most $n$ in $\alpha$ and $\beta$.
And $P^{\a,\b}_n(x)$ is a polynomial of degree at most $n$ in $x$
(the degree is equal to $n$ if e.g.\ $\a,\b > -1$; for more details, see \cite[Chapter IV, Section 4.22]{Sz}).

The formulas \eqref{v1} and \eqref{diff_Jac} hold for all complex parameters (understood as appropriate limits whenever necessary).
Moreover, the $P_n^{\a,\b}$ are eigenfunctions of the Jacobi differential operator given by
$$
J^{\a,\b} = -\big(1-x^2\big) \frac{d^2}{dx^2} - \big[ \b-\a -(\a+\b+2)x\big] \frac{d}{dx}.
$$
More precisely, one has
\begin{equation} \label{ode}
J^{\a,\b} P_n^{\a,\b} = n(n+\a+\b+1) P_n^{\a,\b}, \qquad n \ge 0.
\end{equation}

\subsection*{Measures and some integrals}
We observe  that
$$
\varrho_{\a,\b}\big([-1,1]\big) = h_0^{\a,\b} = 2^{\a+\b+1}\, \frac{\Gamma(\a+1)\Gamma(\b+1)}{\Gamma(\a+\b+2)}, \qquad \a,\b > -1.
$$
For complex $\a$ such that $\Re \a > -1$ and $\a \neq -1/2$, define
$$
\Pi_{\a}(u) = \frac{\Gamma(\a+1)}{\sqrt{\pi} \Gamma(\a+1/2)} \int_0^{u} \big(1-w^2\big)^{\a-1/2}\, dw.
$$
Notice that $\Pi_{\a}$ is an odd function in $-1 < u < 1$, increasing if $\a > -1/2$ and decreasing if $\a < -1/2$.

For real $\a > -1/2$ the measure $d\Pi_{\a}$ is a probability measure in the interval $[-1,1]$,
as seen by changing variable to $v = w^2$ in the integral above.
As $\a \searrow -1/2$, the weak limit of $d\Pi_{\a}$ is
$$
d\Pi_{-1/2} := \frac{\delta_{-1} + \delta_1}2,
$$
where $\delta_{\pm 1}$ is a unit point mass at $\pm 1$.

In general, for all complex $\a$ such that $\Re \a > -1$, $\a \neq -1/2$, the distribution derivative $d\Pi_{\a}$ is
a local complex measure in $(-1,1)$. For real $\a \in (-1,-1/2)$ its density is negative, even, and not integrable in $(-1,1)$.
Furthermore, we have the following lemma.
\begin{lemma}[{\cite[Lemma 2.2]{NSS0}}] \label{lem:mes}
Let $\a \in (-1,-1/2)$ be fixed. Then the density $|\Pi_{\a}(u)|$ defines a finite measure on $[-1,1]$ and
$$
|\Pi_{\a}(u)| \simeq |u|\big(1-|u|\big)^{\a+1/2} \simeq |u|\frac{d\Pi_{\a+1}(u)}{du}, \qquad u \in (-1,1).
$$
\end{lemma}

We define
$$
\mathcal{A} := \big\{ (\a,\b) \in \mathbb{C}^2 : \Re\a > -1, \Re\b > -1 \big\}.
$$
Observe that, see \cite[Section 2]{NSS0}, if $\phi$ is a continuous function in $[0,1]$ satisfying
$|\phi(u)| \le C (1-u)$, with a constant $C = C(\a)$ that is locally bounded, then the integral
$$
I(\a) = \int_{(0,1]} \phi(u) \, d\Pi_{\a}(u)
$$
is well defined and  analytic in $\{\a : \Re\a > -1\}$. Moreover,
$I(-1/2)=0$, as seen from  the bound $|I(\a)| \lesssim |\a+1/2|\int_0^1 (1-u^2)^{\Re\a + 1/2}du$.
This observation generalizes as follows, cf.\ \cite[Section~2]{NSS0}.
\begin{proposition} \label{prop:I}
\qquad
\begin{itemize}
\item[(a)]
Assume that $\phi_{\a,\b}(u)$ is continuous in $(u,\a,\b) \in [0,1]\times \mathcal{A}$
and analytic in $(\a,\b) \in \mathcal{A}$ for each $u \in (0,1)$ and satisfies
 $|\phi_{\a,\b}(u)| \le C(1-u)$, where the constant $C = C(\a,\b)$ is locally bounded in $\mathcal{A}$. Then                         
$$
I(\a,\b) = \int_{(0,1]} \phi_{\a,\b}(u)\, d\Pi_{\a}(u)
$$
is well defined and analytic in $(\a,\b) \in \mathcal{A}$. Moreover, $I(-1/2,\b) = 0$.
\item[(b)]
Let $\phi_{\a,\b}(u,v)$ be continuous in $(u,v,\a,\b) \in [0,1]^2\times\mathcal{A}$
and analytic in $(\a,\b) \in \mathcal{A}$ for each
$u,v \in (0,1)$, and assume that  $|\phi_{\a,\b}(u,v)| \le C(1-u)(1-v)$ with $C$ as in {\rm{(a)}}.
Then the double integral
$$
I(\a,\b) = \iint_{(0,1]^2} \phi_{\a,\b}(u,v) \, d\Pi_{\a}(u)\, d\Pi_{\b}(v)
$$
is well defined and analytic in $(\a,\b) \in \mathcal{A}$. Furthermore, $I(-1/2,\b) = I(\a,-1/2) = 0$.
\end{itemize}
\end{proposition}

\section{Dijksma-Koornwinder formula} \label{sec:DK}

In \cite{DK} Dijksma and Koornwinder found a symmetric \emph{product formula} for two Jacobi polynomials,
which was originally stated as follows (cf.\ \cite[(2.5)]{DK}):
\begin{align}   \label{DijKo}
& P_n^{\a,\b}\big(1-2s^2\big) P_n^{\a,\b}\big(1-2t^2\big) =
\frac{\Gamma(\a+\b+1)\Gamma(n+\a+1)\Gamma(n+\b+1)}{\pi n! \Gamma(n+\a+\b+1) \Gamma(\a+1/2)\Gamma(\b+1/2)} \\
& \quad \times \int_{-1}^1 \int_{-1}^1 C_{2n}^{\a+\b+1}\Big( u s t + v \sqrt{1-s^2}\sqrt{1-t^2}\Big)
	\big(1-u^2\big)^{\a-1/2} \big(1-v^2\big)^{\b-1/2}\, du\, dv.   \nonumber
\end{align}
Here $-1\le s,t \le 1$, and
the formula is valid for $n \ge 0$ and all complex $\a,\b$ such that $\Re \a > -1/2$ and $\Re \b > -1/2$.
The proof in \cite{DK} is group theoretic.
An analytic proof was given slightly later by Koornwinder \cite{K}.
Yet another proof can be found in the more recent paper \cite{Neu}.

We will rewrite and extend the product formula to all $\a,\b$ such that $\Re \a > -1$, $\Re \b > -1$,
i.e., $(\a,\b) \in \mathcal{A}$; see Theorem \ref{thm:DK_ext} below.
For this purpose we define, to begin with  for $\a,\b > -1$ such that $\a+\b+1/2 > -1$ and for $\lambda>-1$,
$$
\mathfrak{C}_n^{\a,\b} := \frac{h_0^{\a,\b}}{(\a+\b+3/2)h_0^{\a+\b+1/2}h_{n}^{\a,\b}},  \quad
n \ge 0,   \qquad  \quad  \mathrm{and}         \qquad \quad \;
\mathfrak{D}_k^{\lambda} := \frac{P_k^{\lambda}(1)}{(\lambda+1)h_k^{\lambda}},  \quad k \ge 0.
$$

In what follows we use implicitly the duplication formula for the gamma function, written as
$$
\Gamma(2\lambda+2) = \Gamma(\lambda+1) \Gamma(\lambda+3/2) 2^{2\lambda+1}/\sqrt{\pi}.
$$
 By a direct computation we see that
\begin{equation} \label{C}
\mathfrak{C}_n^{\a,\b} =
	\begin{cases}
		\frac{\Gamma(\a+1)\Gamma(\b+1)(2n+\a+\b+1)\Gamma(n+\a+\b+1)\Gamma(n+1)}{\sqrt{\pi}\Gamma(\a+\b+5/2)\Gamma(n+\a+1)\Gamma(n+\b+1)},
			& n \ge 1, \\
		\frac{\Gamma(\a+\b+2)}{\sqrt{\pi}\Gamma(\a+\b+5/2)}, & n=0,
	\end{cases}
\end{equation}
and this provides an analytic extension of each $\mathfrak{C}_n^{\a,\b}$, $n\ge 0$,
to $(\a,\b) \in \mathcal{A}$. Similarly for $\lambda > -1$
\begin{equation} \label{D}
\mathfrak{D}_k^{\lambda} =
	\begin{cases}
		\frac{(2k+2\lambda+1)\Gamma(k+2\lambda+1)}{2^{2\lambda+1}\Gamma(\lambda+2)\Gamma(k+\lambda+1)}, & k \ge 1, \\
		\frac{\Gamma(\lambda+3/2)}{\sqrt{\pi}\Gamma(\lambda+2)}, & k = 0.
	\end{cases}
\end{equation}
The last expression defines an analytic extension of each
$\mathfrak{D}_{2n}^{\lambda}$, $n \ge 0$, to $\Re \lambda > -3/2$.
Thus, from now on, we consider $\mathfrak{C}_{n}^{\a,\b}$ and $\mathfrak{D}_{2n}^{\a+\b+1/2}$ as analytic functions
of $(\a,\b) \in \mathcal{A}$ given by \eqref{C} and \eqref{D}, respectively.
Note that $\mathfrak{C}_{n}^{\a,\b}, \mathfrak{D}_{2n}^{\a+\b+1/2} > 0$ for $\a,\b \in (-1,\infty)$.

We denote
\begin{equation} \label{def_Phi}
\Phi_k^{\lambda}(\theta,\varphi,u,v) :=
\mathfrak{D}_k^{\lambda}
	{P}_{k}^{\lambda}\Big(u \sin\frac{\theta}2\sin\frac{\varphi}2 + v\cos\frac{\theta}2\cos\frac{\varphi}2\Big),
\end{equation}
and let $s = \sin\frac{\theta}2$ and  $t = \sin\frac{\varphi}2$ in \eqref{DijKo}. By means of  \eqref{C} and \eqref{D},
one then finds the following more compact form of the Dijksma-Koornwinder formula \eqref{DijKo}:
\begin{equation} \label{comp}
\mathfrak{C}_n^{\a,\b} {P}_n^{\a,\b}(\cos\theta) {P}_n^{\a,\b}(\cos\varphi) =	
	\iint \Phi_{2n}^{\a+\b+1/2}(\theta,\varphi,u,v) \, d\Pi_{\a}(u)\, d\Pi_{\b}(v),
\quad  0 \le \theta, \varphi \le \pi.
\end{equation}
This holds for   $\Re \a > -1/2$ and $\Re \b > -1/2$, and since Jacobi polynomials are continuous functions of their parameters,
it is also valid when $\a=-1/2$ or $\b=-1/2$.

We now proceed essentially as in \cite[Section 2]{NSS0}.
Taking the even parts of $\Phi_k^{\lambda}(\theta,\varphi,u,v)$ in $u$ and $v$, we define
$$
\Phi_{k,E}^{\lambda}(\theta,\varphi,u,v) := \frac{1}{4} \sum_{\xi,\eta = \pm 1} \Phi_k^{\lambda}(\theta,\varphi,\xi u,\eta v).
$$
Then by \eqref{comp} and for symmetry reasons,  one has for $(\a,\b)$ such that $\Re\a,\Re\b > -1/2$
\begin{equation} \label{compE}
 \mathfrak{C}_n^{\a,\b} {{P}_n^{\a,\b}(\cos\theta) {P}_n^{\a,\b}(\cos\varphi)}
= 4 \iint_{(0,1]^2} \Phi_{2n,E}^{\a+\b+1/2}(\theta,\varphi,u,v)\,
	d\Pi_{\a}(u)\, d\Pi_{\b}(v).
\end{equation}
\begin{lemma} \label{lem:int1}
For all $(\a,\b) \in \mathcal{A}$, $n \ge 0$ and $\theta,\varphi \in [0,\pi]$
\begin{align*}
\mathfrak{C}_n^{\a,\b} {{P}_n^{\a,\b}(\cos\theta) {P}_n^{\a,\b}(\cos\varphi)} & =
4 \iint_{(0,1]^2} \Big[ \Phi_{2n,E}^{\a+\b+1/2}(\theta,\varphi,u,v) - \Phi_{2n,E}^{\a+\b+1/2}(\theta,\varphi,u,1) \\
& \qquad	- \Phi_{2n,E}^{\a+\b+1/2}(\theta,\varphi,1,v) + \Phi_{2n,E}^{\a+\b+1/2}(\theta,\varphi,1,1) \Big]
	\, d\Pi_{\a}(u)\, d\Pi_{\b}(v) \\
& \quad + 2 \int_{(0,1]} \Big[ \Phi_{2n,E}^{\a+\b+1/2}(\theta,\varphi,u,1) - \Phi_{2n,E}^{\a+\b+1/2}(\theta,\varphi,1,1) \Big]
	\, d\Pi_{\a}(u) \\
& \quad + 2 \int_{(0,1]} \Big[ \Phi_{2n,E}^{\a+\b+1/2}(\theta,\varphi,1,v) - \Phi_{2n,E}^{\a+\b+1/2}(\theta,\varphi,1,1) \Big]
	\, d\Pi_{\b}(v) \\
& \quad + \Phi_{2n,E}^{\a+\b+1/2}(\theta,\varphi,1,1).
\end{align*}
\end{lemma}

\begin{proof}
When $\Re\a, \Re\b > -1/2$, the claimed identity  is a straightforward consequence of \eqref{compE},
since the integrals can be evaluated term by term.
The left-hand side of the identity is analytic in $(\a,\b) \in \mathcal{A}$.
It suffices to verify that the right-hand side is analytic in
$\mathcal{A}$, since then the conclusion will follow by  analytic continuation.

The last term of the right-hand side is analytic in $(\a,\b)\in \mathcal{A}$, as easily  seen from \eqref{def_Phi}.
The second term of the right-hand side, i.e., the integral in $u$,
is of the form $I(\a,\b)$ from item (a) of Proposition \ref{prop:I}, with
$$\phi_{\a,\b}(u,v) = \Phi_{2n,E}^{\a+\b+1/2}(\theta,\varphi,u,1) - \Phi_{2n,E}^{\a+\b+1/2}(\theta,\varphi,1,1).$$
Note that $|\phi_{\a,\b}(u)| \le C(1-u)$, since the derivative
$\partial_u \Phi_{2n,E}^{\a+\b+1/2}(\theta,\varphi,u,1)$ is bounded locally uniformly in $\a$ and $\b$,
see \eqref{diff_Jac}. The third term of the right-hand side, i.e., the integral in $v$, is analogous.

Finally, the first term of the right-hand side, i.e., the double integral,
is of the form $I(\a,\b)$ from item (b) of Proposition \ref{prop:I}, with
\begin{align*}
& \phi_{\a,\b}(u,v) \\
& \quad = \Phi_{2n,E}^{\a+\b+1/2}(\theta,\varphi,u,v) - \Phi_{2n,E}^{\a+\b+1/2}(\theta,\varphi,u,1)
		- \Phi_{2n,E}^{\a+\b+1/2}(\theta,\varphi,1,v) + \Phi_{2n,E}^{\a+\b+1/2}(\theta,\varphi,1,1),
\end{align*}
and the condition $|\phi_{\a,\b}(u,v)| \le C(1-u)(1-v)$ is satisfied since the second order derivatives
of $\Phi_{2n,E}^{\a+\b+1/2}(\theta,\varphi,u,v)$ in $u,v$ are bounded locally uniformly in $\a$ and $\b$, cf.\ \eqref{diff_Jac}.

Altogether, it follows that the right-hand side of the identity is analytic in $(\a,\b) \in \mathcal{A}$.
This finishes the proof.
\end{proof}

The next result is the announced extension of the Dijksma-Koornwinder formula. We state it for real values of the parameters.
A variant with complex parameters is indicated in Remark~\ref{rem:complex} below.
For the sake of completeness, the original Dijksma-Koornwinder formula is included as item (i).
By convention, $\Phi_k^{\lambda}(\theta,\varphi,u,v) \equiv 0$ when $k < 0$.
\begin{theorem} \label{thm:DK_ext}
Let $n \ge 0$ and $\theta,\varphi \in [0,\pi]$. Then, with all integrations taken over $[-1,1]^2$,
\begin{itemize}
\item[(i)] If $\a,\b \ge -1/2$, then
$$
\mathfrak{C}_n^{\a,\b} {{P}_n^{\a,\b}(\cos\theta) {P}_n^{\a,\b}(\cos\varphi)}
= \iint \Phi_{2n}^{\a+\b+1/2}(\theta,\varphi,u,v)\, d\Pi_{\a}(u) \, d\Pi_{\b}(v).
$$
\item[(ii)] If $-1 < \b < -1/2 \le \a$, then
\begin{align*}
& \mathfrak{C}_n^{\a,\b} {{P}_n^{\a,\b}(\cos\theta) {P}_n^{\a,\b}(\cos\varphi)}\\ &
\quad = \iint \Big\{ -2(\a+\b+5/2) \cos\frac{\theta}2\cos\frac{\varphi}2\, \Phi_{2n-1}^{\a+\b+3/2}(\theta,\varphi,u,v)\,
	d\Pi_{\a}(u)\, \Pi_{\b}(v)\, dv \\
& \qquad \qquad + \Phi_{2n}^{\a+\b+1/2}(\theta,\varphi,u,v) \, d\Pi_{\a}(u)\, d\Pi_{-1/2}(v) \Big\}.
\end{align*}
\item[(iii)] If $-1 < \a < -1/2 \le \b$, then
\begin{align*}
& \mathfrak{C}_n^{\a,\b} {{P}_n^{\a,\b}(\cos\theta) {P}_n^{\a,\b}(\cos\varphi)}\\ &
\quad = \iint \Big\{ -2(\a+\b+5/2) \sin\frac{\theta}2\sin\frac{\varphi}2\, \Phi_{2n-1}^{\a+\b+3/2}(\theta,\varphi,u,v)\,
	\Pi_{\a}(u)\, du \, d\Pi_{\b}(v) \\
& \qquad \qquad + \Phi_{2n}^{\a+\b+1/2}(\theta,\varphi,u,v) \, d\Pi_{-1/2}(u)\, d\Pi_{\b}(v) \Big\}.
\end{align*}
\item[(iv)] If $-1 < \a,\b < -1/2$, then
\begin{align*}
& \mathfrak{C}_n^{\a,\b} {{P}_n^{\a,\b}(\cos\theta) {P}_n^{\a,\b}(\cos\varphi)}\\ &
\quad = \iint \Big\{ (\a+\b+5/2)(\a+\b+7/2)\sin\theta\sin\varphi \, \Phi_{2n-2}^{\a+\b+5/2}(\theta,\varphi,u,v)\,
	\Pi_{\a}(u)\, du \, \Pi_{\b}(v)\, dv \\
& \qquad \qquad - 2(\a+\b+5/2) \sin\frac{\theta}2\sin\frac{\varphi}2\, \Phi_{2n-1}^{\a+\b+3/2}(\theta,\varphi,u,v)\,
	\Pi_{\a}(u)\, du \, d\Pi_{-1/2}(v) \\
& \qquad \qquad -2(\a+\b+5/2) \cos\frac{\theta}2\cos\frac{\varphi}2\, \Phi_{2n-1}^{\a+\b+3/2}(\theta,\varphi,u,v)\,
	d\Pi_{-1/2}(u)\, \Pi_{\b}(v)\, dv \\
& \qquad \qquad + \Phi_{2n}^{\a+\b+1/2}(\theta,\varphi,u,v)\, d\Pi_{-1/2}(u)\, d\Pi_{-1/2}(v) \Big\}.
\end{align*}
\end{itemize}
\end{theorem}

\begin{proof}
Item (i) is a restatement of \eqref{comp}.
The proofs of the remaining items rely on combining Lemmas \ref{lem:int1} and \ref{lem:mes} with symmetries
of the quantity $\Phi_{k,E}^{\lambda}(\theta,\varphi,u,v)$ and the measures involved. Further, the following
formulas for derivatives of $\Phi_k^{\lambda}$ are useful; they follow from \eqref{diff_Jac}
and the identity
$\frac12\,(k+2\lambda +1)\, \mathfrak{D}_k^{\lambda}  =2(\lambda + 2) \, \mathfrak{D}_{k-1}^{\lambda+1}, \; k \ge 1$:
\begin{align*}
\partial_{u} \Phi_{k}^{\lambda}(\theta,\varphi,u,v) & = 2(\lambda+2) \sin\frac{\theta}2\sin\frac{\varphi}2 \,
	\Phi_{k-1}^{\lambda+1}(\theta,\varphi,u,v), \\
\partial_{v} \Phi_{k}^{\lambda}(\theta,\varphi,u,v) & = 2(\lambda+2) \cos\frac{\theta}2\cos\frac{\varphi}2 \,
	\Phi_{k-1}^{\lambda+1}(\theta,\varphi,u,v), \\
\partial_{u} \partial_{v} \Phi_{k}^{\lambda}(\theta,\varphi,u,v) & = (\lambda+2)(\lambda+3) \sin\theta\sin\varphi \,
	\Phi_{k-2}^{\lambda+2}(\theta,\varphi,u,v).
\end{align*}
The details are analogous to those in the proof of \cite[Proposition 2.3]{NSS0} and thus omitted.
\end{proof}

\begin{remark} \label{rem:complex}
To state Theorem \ref{thm:DK_ext} for complex values of the parameters, let
$$
\mathcal{B}_1 := \{\a \in \mathbb{C} : \Re\a > -1/2\} \cup\{-1/2\},
	\qquad \mathcal{B}_2 := \{\a\in\mathbb{C}: \Re\a > -1\} \setminus \mathcal{B}_1.
$$
Then the cases should be defined, respectively, by (i) $\a,\b \in \mathcal{B}_1$,
(ii) $\a \in \mathcal{B}_1$, $\b \in \mathcal{B}_2$, (iii) $\a\in\mathcal{B}_2$, $\b\in\mathcal{B}_1$,
(iv) $\a,\b \in \mathcal{B}_2$, with the formulas unchanged.
\end{remark}

\section{Reduction formula} \label{sec:red}

Recall that for $\a,\b > -1$, $x,y \in [-1,1]$ and $t > 0$ the Jacobi heat kernel $G_t^{\a,\b}(x,y)$ is given by \eqref{ker_Jac}.
For basic information concerning this kernel, see e.g.\ \cite{NoSj}.
We will use the following lemma from our earlier paper \cite{NSS2}.
\begin{lemma}[{\cite[Lemma 3.3]{NSS2}}] \label{lem:diff}
Let $\a,\b > -1$. Then
\begin{align*}
\frac{d}{dx}\, G_t^{\a,\b}(x,1)
=
2 (\a+1)\, e^{-t(\a+\b+2)}\, G_t^{\a+1,\b+1}(x,1),
\qquad x \in [-1,1], \quad t > 0.
\end{align*}
In particular, the function $x \mapsto G_t^{\a,\b}(x,1)$ is strictly increasing on $[-1,1]$.
\end{lemma}

We now introduce an auxiliary function $H_t^{\lambda}(x)$ which can be regarded as an extension
for $-3/2<\lambda \le -1$ of (a constant times) the even part in $x$ of the Jacobi-ultraspherical heat kernel $G_t^{\lambda}(x,1)$.
Let
\begin{equation*}
H_t^{\lambda}(x) := \sum_{n=0}^{\infty} e^{-t 2n(2n+2\lambda+1)} \,\mathfrak{D}_{2n}^{\lambda} \,P_{2n}^{\lambda}(x), \qquad
	x \in [-1,1], \quad t >0, \quad -3/2 < \lambda \le -1.
\end{equation*}
Note that $H_t^{\lambda}(x)$ is well defined. Indeed, the constants $\mathfrak{D}_{2n}^{\lambda}$ are well defined by
\eqref{D}. By Stirling's formula, one has $\mathfrak{D}_{2n}^{\lambda} = \mathcal{O}(n^{\lambda+1})$ as $n \to \infty$.
Furthermore, the Jacobi polynomials satisfy (see \cite[Chapter VII, Section 7.32, Theorem 7.32.2]{Sz})
$$
\big| P_k^{\a,\b}(x)\big| \lesssim (k+1)^{\max\{\a,\b,-1/2\}}, \qquad k \ge 0, \quad x \in [-1,1],
$$
for any fixed $\a,\b \in \mathbb{R}$. Consequently, we see that the series defining $H_t^{\lambda}(x)$
converges uniformly in $(t,x) \in [T,\infty)\times [-1,1]$ for each fixed $T>0$. Using \eqref{diff_Jac} we also
see that the series can be differentiated term by term in $x$ and $t$ arbitrarily many times.
Thus $H_t^{\lambda}(x)$ is a smooth function of $(t,x) \in (0,\infty)\times [-1,1]$.

Next, we observe that $H_t^{\lambda}$ satisfies the heat equation based on the Jacobi-ultraspherical operator $J^{\lambda}$,
\begin{equation} \label{Hlheat}
\Big( \frac{d}{dt} + J^{\lambda}\Big) H_t^{\lambda}(x) = 0, \qquad x \in [-1,1], \quad t >0.
\end{equation}
This follows from the fact that each term of the defining series satisfies this equation, see \eqref{ode}.
Moreover, we have the following.
\begin{lemma} \label{lem:diff_H}
Let $-3/2 < \lambda \le -1$. Then
$$
\frac{d}{dx}\, H_t^{\lambda}(x) = 2e^{-t(2\lambda+2)}\Big[ G_t^{\lambda+1}(x,1)\Big]_{\emph{odd}}, \qquad x \in [-1,1], \quad t > 0,
$$
where the subscript `odd' indicates the odd part of the function in $x$. Similarly,
$$
\frac{d^2}{dx^2}\, H_t^{\lambda}(x) = 4(\lambda+2)e^{-t(4\lambda+6)}
	\Big[ G_t^{\lambda+2}(x,1)\Big]_{\emph{even}}, \qquad x \in [-1,1], \quad t > 0,
$$
where the subscript `even' indicates the even part of the function in $x$.
\end{lemma}

\begin{proof}
To verify the first identity,
differentiate the series defining $H_t^{\lambda}(x)$ term by term and use \eqref{diff_Jac} together with the fact that
ultraspherical polynomials of odd orders are odd functions, and those of even orders are even functions.
The second identity follows from the first one and Lemma \ref{lem:diff}.
We leave the straightforward computations to the reader.
\end{proof}

The next result extends a \emph{reduction formula} for the Jacobi heat kernel that was obtained in \cite[Theorem 3.1]{NoSj}
by means of the Dijksma-Koornwinder product formula. For the sake of completeness,
the result from \cite{NoSj} is stated as item (i) of Theorem \ref{thm:red_ext}.
\begin{theorem} \label{thm:red_ext}
Let $t > 0$ and $\theta, \varphi \in [0,\pi]$. Then, with all integrations taken over $[-1,1]^2$,
\begin{itemize}
\item[(i)] If $\a,\b \ge -1/2$, then
\begin{align*}
& \frac{h_0^{\a,\b}}{h_0^{\a+\b+1/2}}  \,G_t^{\a,\b}(\cos\theta,\cos\varphi) \\
& \quad = \iint G_{t/4}^{\a+\b+1/2}\Big( u \sin\frac{\theta}2\sin\frac{\varphi}2
		+ v \cos\frac{\theta}2\cos\frac{\varphi}2, 1 \Big) \, d\Pi_{\a}(u) \, d\Pi_{\b}(v).
\end{align*}
\item[(ii)] If $-1 < \b < -1/2 \le \a$, then
\begin{align*}
& \frac{h_0^{\a,\b}}{h_0^{\a+\b+1/2}}\, G_t^{\a,\b}(\cos\theta,\cos\varphi) \\
& \quad = \iint \Big\{ -2(\a+\b+3/2) e^{-t(\a+\b+3/2)/2} \cos\frac{\theta}2\cos\frac{\varphi}2 \\
& \qquad \qquad \qquad \times	G_{t/4}^{\a+\b+3/2}\Big(u \sin\frac{\theta}2\sin\frac{\varphi}2 +
	v\cos\frac{\theta}2\cos\frac{\varphi}2,1\Big)
	\, d\Pi_{\a}(u) \, \Pi_{\b}(v)\, dv \\
& \qquad \qquad + G_{t/4}^{\a+\b+1/2}\Big(u \sin\frac{\theta}2\sin\frac{\varphi}2 + v\cos\frac{\theta}2\cos\frac{\varphi}2,1\Big)
	\, d\Pi_{\a}(u)\, d\Pi_{-1/2}(v) \Big\}.
\end{align*}
\item[(iii)] If $-1 < \a < -1/2 \le \b$, then
\begin{align*}
& \frac{h_0^{\a,\b}}{h_0^{\a+\b+1/2}}\, G_t^{\a,\b}(\cos\theta,\cos\varphi) \\
& \quad = \iint \Big\{ -2(\a+\b+3/2) e^{-t(\a+\b+3/2)/2} \sin\frac{\theta}2\sin\frac{\varphi}2 \\
& \qquad \qquad \qquad \times	G_{t/4}^{\a+\b+3/2}\Big(u \sin\frac{\theta}2\sin\frac{\varphi}2 +
	v\cos\frac{\theta}2\cos\frac{\varphi}2,1\Big)
	\, \Pi_{\a}(u)\, du \, d\Pi_{\b}(v) \\
& \qquad \qquad + G_{t/4}^{\a+\b+1/2}\Big(u \sin\frac{\theta}2\sin\frac{\varphi}2 + v\cos\frac{\theta}2\cos\frac{\varphi}2,1\Big)
	\, d\Pi_{-1/2}(u)\, d\Pi_{\b}(v) \Big\}.
\end{align*}
\item[(iv)] If $-1 < \a,\b < -1/2$ and $\a+\b > -3/2$, then
\begin{align*}
& \frac{h_0^{\a,\b}}{h_0^{\a+\b+1/2}}\, G_t^{\a,\b}(\cos\theta,\cos\varphi) \\
& \quad = \iint \Big\{ (\a+\b+3/2)(\a+\b+5/2) e^{-t(\a+\b+2)} \sin\theta\sin\varphi \\
& \qquad \qquad \qquad \times	G_{t/4}^{\a+\b+5/2}\Big(u \sin\frac{\theta}2\sin\frac{\varphi}2 +
	v\cos\frac{\theta}2\cos\frac{\varphi}2,1\Big)
 \, \Pi_{\a}(u)\, du \, \Pi_{\b}(v)\, dv \\
& \qquad \qquad -2(\a+\b+3/2) e^{-t(\a+\b+3/2)/2} \sin\frac{\theta}2\sin\frac{\varphi}2 \\
& \qquad \qquad \qquad \times	G_{t/4}^{\a+\b+3/2}\Big(u \sin\frac{\theta}2\sin\frac{\varphi}2 +
	v\cos\frac{\theta}2\cos\frac{\varphi}2,1\Big)
	\, \Pi_{\a}(u)\, du \, d\Pi_{-1/2}(v) \\
& \qquad \qquad -2(\a+\b+3/2) e^{-t(\a+\b+3/2)/2} \cos\frac{\theta}2\cos\frac{\varphi}2 \\
& \qquad \qquad \qquad \times	G_{t/4}^{\a+\b+3/2}\Big(u \sin\frac{\theta}2\sin\frac{\varphi}2 +
	v\cos\frac{\theta}2\cos\frac{\varphi}2,1\Big)
	\, d\Pi_{-1/2}(u) \, \Pi_{\b}(v)\, dv \\
& \qquad \qquad + G_{t/4}^{\a+\b+1/2}\Big(u \sin\frac{\theta}2\sin\frac{\varphi}2 + v\cos\frac{\theta}2\cos\frac{\varphi}2,1\Big)
	\, d\Pi_{-1/2}(u)\, d\Pi_{-1/2}(v) \Big\}.
\end{align*}
\item[(v)] If $-1 < \a,\b < -1/2$ and $\a+\b \le -3/2$, then
\begin{align*}
& h_0^{\a,\b} \mathfrak{C}_0^{\a,\b} G_t^{\a,\b}(\cos\theta,\cos\varphi) \\
& \quad = \iint \Big\{(\a+\b+5/2) e^{-t(\a+\b+2)} \sin\theta\sin\varphi \\
& \qquad \qquad \qquad \times	G_{t/4}^{\a+\b+5/2}\Big(u \sin\frac{\theta}2\sin\frac{\varphi}2 +
	v\cos\frac{\theta}2\cos\frac{\varphi}2,1\Big)
 \, \Pi_{\a}(u)\, du \, \Pi_{\b}(v)\, dv \\
& \qquad \qquad -2 e^{-t(\a+\b+3/2)/2} \sin\frac{\theta}2\sin\frac{\varphi}2 \\
& \qquad \qquad \qquad \times	G_{t/4}^{\a+\b+3/2}\Big(u \sin\frac{\theta}2\sin\frac{\varphi}2 +
	v\cos\frac{\theta}2\cos\frac{\varphi}2,1\Big)
	\, \Pi_{\a}(u)\, du \, d\Pi_{-1/2}(v) \\
& \qquad \qquad -2 e^{-t(\a+\b+3/2)/2} \cos\frac{\theta}2\cos\frac{\varphi}2 \\
& \qquad \qquad \qquad \times	G_{t/4}^{\a+\b+3/2}\Big(u \sin\frac{\theta}2\sin\frac{\varphi}2 +
	v\cos\frac{\theta}2\cos\frac{\varphi}2,1\Big)
	\, d\Pi_{-1/2}(u) \, \Pi_{\b}(v)\, dv \\
& \qquad \qquad + H_{t/4}^{\a+\b+1/2}\Big(u \sin\frac{\theta}2\sin\frac{\varphi}2 + v\cos\frac{\theta}2\cos\frac{\varphi}2 \Big)
	\, d\Pi_{-1/2}(u)\, d\Pi_{-1/2}(v) \Big\}.
\end{align*}
\end{itemize}
\end{theorem}

\begin{proof}
To verify items (i)--(iv) above, one  starts by multiplying  the identities in (i)--(iv) of Theorem \ref{thm:DK_ext}
by the factor $\exp(-tn(n+\a+\b+1))$. The resulting equations are summed over
$n=0,1,2,\ldots$. Then one uses the simple relations
\begin{align*}
& tn(n+\a+\b+1) \\
& \quad = (t/4) 2n [2n + (\a+\b+1/2) + (\a+\b+1/2) + 1] \\
& \quad = (t/4) (2n-1) [2n-1 +(\a+\b+3/2) + (\a+\b+3/2) + 1] + t(\a+\b+3/2)/2 \\
& \quad = (t/4) (2n-2) [2n-2 + (\a+\b+5/2) + (\a+\b+5/2) + 1] + t(\a+\b+2)
\end{align*}
and the fact that Jacobi-ultraspherical polynomials of even (odd) orders are even (odd) functions.
It follows that the series obtained represents in each case
either the even or the odd part with respect to the first variable of the Jacobi-ultraspherical
heat kernel with appropriate parameter and the second variable fixed at the right endpoint.
Since, for symmetry reasons, the corresponding complementary odd/even parts give no contributions to the integrals,
we end up with the desired conclusions.

Dealing with the last term in item (v) is even more straightforward, since it requires only some of the above arguments.
\end{proof}

\section{{Preliminaries for the proof of Theorem \ref{thm:main}}} \label{sec:bnd}

\subsection*{Initial estimates}

The next result is a special case of Theorem \ref{thm:main}.
\begin{lemma} \label{lem:jac_low}
Let $\lambda > -1$. Then
$$
G_t^{\lambda}(\cos\theta,1) \simeq t^{-\lambda-1} \exp\bigg( -\frac{\theta^2}{4t}\bigg),
$$
uniformly in $0 \le \theta \le \pi/2$ and $0 < t \le 1$.
\end{lemma}

\begin{remark} \label{rem:s}
This estimate is new for $\lambda < -1/2$, but for $\lambda \ge -1/2$ it is contained in the known part of Theorem \ref{thm:main}.
\end{remark}

\begin{proof}[{Proof of Lemma \ref{lem:jac_low}}]
We first verify that
$\lim_{t\to 0^+} G_t^{\lambda}(0,1) = 0$.
The heat equation yields            
\begin{equation*}
 \frac \partial {\partial t} \, G_t^{\lambda}(0,1)  = - J^\lambda\, G_t^{\lambda}(x,1)\Big|_{x=0} 
	=  \frac {d^2} {d x^2}\, G_t^{\lambda}(x,1)\Big|_{x=0}.
\end{equation*}
To estimate this second derivative, we apply Lemma \ref{lem:diff} twice and then use Remark \ref{rem:s}, getting
\begin{equation} \label{dGdt}
  \frac \partial {\partial t} \, G_t^{\lambda}(0,1)
\simeq G_t^{\lambda+2}(0,1) \simeq  t^{-\lambda-3}\,\exp\bigg(-\frac{\pi^2}{16t} \bigg) >0,  \qquad 0<t\le 1.
\end{equation}
Thus $G_t^{\lambda}(0,1)$ is a non-negative, increasing function of $t$ in $(0,1]$, and so the limit
$$
\eta := \lim_{t\to 0^+} G_t^{\lambda}(0,1)
$$
exists and is non-negative.

Aiming at a contradiction, we assume that $\eta>0$. Since $ G_t^{\lambda}(x,1)$ is increasing in $x$ by Lemma~\ref{lem:diff},
we then have $G_t^{\lambda}(x,1) \ge G_t^{\lambda}(0,1) > \eta$ for $0 \le x \le 1$ and $0< t \le 1$,
and by symmetry $G_t^{\lambda}(1,y) > \eta$ for $0 \le y \le 1$. Now let $\big(T_t^{\lambda}\big)$
be the Jacobi (ultraspherical) semigroup
and take a nontrivial continuous function $f \ge 0$ supported in a compact subinterval of $(0,1)$. Then we get
\begin{equation*}
 T^{\lambda}_tf(1) = \int_{-1}^{1}G_t^{\lambda}(1,y)f(y)\,d\varrho_{\lambda}(y) > \eta\,\|f\|_{L^1(d\varrho_{\lambda})} > 0
\end{equation*}
for $0<t\le 1$. But $\big(T^{\lambda}_t\big)$ is a symmetric diffusion semigroup in the sense of
\cite[Chapter~3]{topics} and its maximal operator is $L^{\infty}$-bounded.
By standard arguments, see e.g.\ \cite[Chapter~2, Section~2]{Duo}, it follows that
$T^{\lambda}_t f$ converges to $f$ as $t \to 0^+$, uniformly in $[-1,1]$, since this convergence certainly holds for the
dense subset of $C([-1,1])$ consisting of finite linear combinations of Jacobi-ultraspherical polynomials.
This is a contradiction which proves that $\eta = 0$.

Then \eqref{dGdt} and Remark \ref{rem:s} imply for $0<t \le 1$
\begin{align} \label{G01}
 G_t^{\lambda}(0,1) & \simeq
 \int_{0}^{t} \tau^{-\lambda-3} \,\exp\bigg(-\frac{\pi^2}{16\tau} \bigg) \,d\tau \simeq
 \int_{1/t}^{\infty}  \sigma^{\lambda+1} \,\exp\bigg(-\frac{\pi^2}{16}\,\sigma \bigg) \,d\sigma \notag \\ &\simeq
 t^{-\lambda-1}\,\exp\bigg(-\frac{\pi^2}{16t} \bigg),
\end{align}
where we  made the change of variable $\sigma = 1/\tau$.

For $x = \cos \theta \in (0,1]$, we apply first Lemma \ref{lem:diff} and then Remark \ref{rem:s}, to get
\begin{align*}
 G_t^{\lambda}(x,1) - G_t^{\lambda}(0,1) & =  \int_{0}^{x} \frac {d} {dy}\, G_t^{\lambda}(y,1)\,dy \simeq
     \int_{0}^{x}   G_t^{\lambda+1}(y,1)\,dy  \\
& \quad = \int_{\theta}^{\pi/2}  G_t^{\lambda+1}(\cos\varphi,1)\,\sin\varphi\,d\varphi \simeq
  t^{-\lambda-2}\,   \int_{\theta}^{\pi/2}  \,\exp\bigg(-\frac{\varphi^2}{4t} \bigg)\,\varphi\,d\varphi\\
& \quad = 2 \, t^{-\lambda-1}\,\bigg[\exp\bigg(-\frac{\theta^2}{4t} \bigg) - \exp\bigg(-\frac{\pi^2}{16t} \bigg)\bigg].
\end{align*}
From this and \eqref{G01}, we can finish the proof by separating the case $((\pi/2)^2-\theta^2)/t > 1$ and the contrary case.
\end{proof}

\begin{lemma}\label{Htl}
Let $\lambda \in (-3/2,-1]$. Then
$$
 H_{t}^{\lambda}(\cos\theta) \simeq  t^{-\lambda -1}\,\exp\bigg( -\frac{\theta^2}{4t} \bigg),
$$
uniformly in $0\le \theta \le \pi/2$ and $0<t \le 1$.
\end{lemma}

\begin{proof}
We  first claim that
 $$
H_{t}^{\lambda}(0) \to 0 \qquad \mathrm{as} \qquad t \to 0^+.
$$
In the  formula of Theorem \ref{thm:red_ext}\,(v), we take $\a=\b=\lambda/2-1/4>-1$, $\varphi = 0$ and  $\theta =\pi$.
Then the fourth integral in the right-hand side of this formula is precisely $H_{t/4}^{\lambda}(0)$.
Moreover, the first three integrals in the right-hand side all vanish because of the trigonometric factors.
In the left-hand side, we have constant times $G_{t}^{\a}(-1,1)$, and this quantity is non-negative and,
in view of Lemma \ref{lem:diff}, no larger than $G_{t}^{\a}(0,1)$. Lemma \ref{lem:jac_low} shows that  $G_{t}^{\a}(0,1) \to 0$ as
$t \to 0^+$, and the claim follows.

We will next estimate  $H_{t}^{\lambda}(x)$ for $0 \le x \le 1$ and $0<t \le 1$, by integrating its derivatives.
Letting first  $x = 0$, we conclude
from \eqref{Hlheat} and Lemma \ref{lem:diff_H} that
\begin{equation*}
\frac \partial {\partial t} \, H_{t}^{\lambda}(0) = - J^\lambda H_{t}^{\lambda}(x)\Big|_{x=0} =
	\frac{d^2}{d x^2}\, H_{t}^{\lambda}(x)\Big|_{x=0} \simeq G_t^{\lambda+2}(0,1) \simeq t^{-\lambda -3} \,
		\exp\bigg( -\frac{\pi^2}{16t}\bigg),
\end{equation*}
the last relation by Lemma \ref{lem:jac_low}. Integrating as in \eqref{G01}, we get for $0<t \le 1$
\begin{equation} \label{H0}
  H_{t}^{\lambda}(0) \simeq t^{-\lambda -1}\,\exp\bigg( -\frac{\pi^2}{16t}\bigg).
\end{equation}

Let now $0 \le x \le 1$ and $0<t \le 1$. We write, using Lemma \ref{lem:diff_H} and then Lemma \ref{lem:diff},
\begin{align*} 
H_{t}^{\lambda}(x) & =  H_{t}^{\lambda}(0) + \int_{0}^{x} \frac{d}{dy}\,H_{t}^{\lambda}(y)\,dy
= H_{t}^{\lambda}(0) + e^{-t(2\lambda+2)}\, \int_{0}^{x} \Big[G_t^{\lambda+1}(y,1) - G_t^{\lambda+1}(-y,1)\Big]\,dy \notag \\
& \simeq  H_{t}^{\lambda}(0) +  \int_{0}^{x} \, \int_{-y}^{y}G_t^{\lambda+2}(z,1)\,dz\,dy
\simeq   H_{t}^{\lambda}(0) +  \int_{0}^{x} \, \int_{0}^{y}G_t^{\lambda+2}(z,1)\,dz\,dy,
\end{align*}
the last step by the monotonicity of $G_t^{\lambda+2}(\cdot,1)$. Now let $x= \cos \theta, \; y = \cos \varphi$ and $z= \cos \kappa$. 
From Lemma \ref{lem:jac_low}, we see that the last double integral is of order of magnitude
\begin{align} \label{double}
& \int^{\pi/2}_{\theta} \int^{\pi/2}_{\varphi} G_t^{\lambda+2}(\cos\kappa,1)\sin\kappa\, d\kappa\, \sin\varphi\, d\varphi
	\simeq t^{-\lambda -3}\, \int^{\pi/2}_{\theta} \int^{\pi/2}_{\varphi} \exp\bigg( -\frac{\kappa^2}{4t}\bigg)
		\kappa\, d\kappa \,\varphi\, d\varphi \\
& \qquad = 2\, t^{-\lambda -2}\, \int^{\pi/2}_{\theta}
 \bigg[\exp\bigg( -\frac{\varphi^2}{4t} \bigg) - \exp\bigg( -\frac{\pi^2}{16t} \bigg)\bigg] \,\varphi\, d\varphi  \notag \\
& \qquad = 4\, t^{-\lambda -1}\, \bigg[\exp\bigg(-\frac{\theta^2}{4t}\bigg) - \exp\bigg( -\frac{\pi^2}{16t}\bigg)\bigg]     
  - t^{-\lambda -2}\,\bigg[\Big(\frac{\pi}2\Big)^2 -\theta^2\bigg]\,\exp\bigg( -\frac{\pi^2}{16t}\bigg) \notag \\
& \qquad = 4\, t^{-\lambda -1}\, \exp\bigg( -\frac{\theta^2}{4t} \bigg)\bigg[1 - \bigg(1 +
	\frac{(\pi/2)^2 -\theta^2 }{4t} \bigg)\, \exp\bigg( -\frac{(\pi/2)^2 -\theta^2 }{4t} \bigg)\bigg]. \notag
\end{align}

Now if ${((\pi/2)^2 -\theta^2) }/{4t} > C$ for some suitably large constant $C$,
the product inside the square bracket in the last expression will be small,
and then the whole expression is of order of magnitude $t^{-\lambda-1}\exp( -\theta^2/{4t})$.
Together with  \eqref{H0}, this implies the conclusion of the lemma.
And in the contrary case ${((\pi/2)^2 -\theta^2) }/{4t} \le C$, one has
$t^{-\lambda-1}\exp( -\theta^2/{4t}) \simeq t^{-\lambda-1}\exp( -{\pi^2}/{16t}) \simeq H_{t}^{\lambda}(0)$ because of \eqref{H0}.
Since the expressions in \eqref{double} are non-negative and no larger than a constant times $t^{-\lambda-1}\exp( -\theta^2/{4t})$,
the conclusion follows from  \eqref{H0}.
\end{proof}

\subsection*{Further preparations}

Let $0 \le \theta,\varphi \le \pi$. For $0 \le u,v \le 1$ we define
$$
 F(u,v) := \bigg[\arccos\bigg(u\sin\frac{\theta}2\sin\frac{\varphi}2 + v\cos\frac{\theta}2\cos\frac{\varphi}2\bigg)\bigg]^2.
$$
Then $0 \le F(1,1) = (\theta - \varphi)^2/4 \le F(u,v) \le \pi^2/4 = F(0,0)$.
\begin{lemma} \label{lem:exp}
There exist absolute constants $c_1 > c_2 > 0$ such that the bounds
\begin{align*}
& \exp\bigg(-c_1\,\frac{(1-u) \theta\varphi}t\bigg)
\exp\bigg(-c_1\,\frac{ (1-v)(\pi - \theta)(\pi - \varphi)}t \bigg)\\ & \qquad \qquad \le 
\exp\bigg(-\frac{ F(u,v)}t\bigg)\, \exp\bigg(\frac{ F(1,1)}t\bigg)  \\ & \qquad \qquad \qquad \qquad \le
\exp\bigg(-c_2\,\frac{(1-u) \theta\varphi}t\bigg)
\exp\bigg(-c_2\,\frac{ (1-v)(\pi - \theta)(\pi - \varphi)}t\bigg)
\end{align*}
hold for $\theta,\varphi \in [0,\pi]$, $u,v \in [0,1]$ and $t > 0$.
\end{lemma}

\begin{proof}
We have $\cos \sqrt{F(u,v)} = u\sin\frac{\theta}2\sin\frac{\varphi}2 + v\cos\frac{\theta}2\cos\frac{\varphi}2$.
Differentiating this equation with respect to $u$ and $v$, we get
\begin{align*}
 -\partial_u F(u,v) & = 2\, \frac{\sqrt{F(u,v)}}{\sin \sqrt{F(u,v)}}\,\sin\frac{\theta}2 \sin\frac{\varphi}2
\simeq \sin\frac{\theta}2\sin\frac{\varphi}2 \simeq \theta\varphi; \\
 -\partial_v F(u,v) & = 2\, \frac{\sqrt{F(u,v)}}{\sin \sqrt{F(u,v)}}\,\cos\frac{\theta}2 \cos\frac{\varphi}2
\simeq \cos\frac{\theta}2\cos\frac{\varphi}2 \simeq (\pi - \theta)(\pi - \varphi).
\end{align*}
Thus for any  $0 \le u, v \le 1$ and $0 \le \theta,\varphi \le \pi$
\begin{equation*}
 F(u,v) - F(1,1) \simeq (1-u) \theta\varphi + (1-v)(\pi - \theta)(\pi - \varphi).
\end{equation*}
Dividing by $t$ and exponentiating, we obtain the desired conclusion.
\end{proof}

We will also need the following.
\begin{lemma}\label{ssigma}
 Let $\gamma \ge -1/2$ be fixed. Then uniformly in $\xi \ge 0$
\begin{equation*}
	\int_{[0,1]} \exp\big(-\xi (1-s) \big) \, d\Pi_{\gamma}(s) \simeq \frac1{(1+\xi)^{\gamma+1/2}}.
\end{equation*}
\end{lemma}

\begin{proof}
The case $\gamma = -1/2$  is trivial, so we assume  $\gamma > -1/2$.
Using   the definition of $\Pi_{\gamma}$ and making the transformation $\sigma = \xi (1-s)$, we get
\begin{align*}
	\int_{[0,1]} \exp\big(-\xi (1-s) \big) \, d\Pi_{\gamma}(s) & \simeq
\int_0^1 \exp\big(-\xi (1-s) \big) \,(1-s)^{\gamma-1/2}\,ds  \\ & =
\xi^{-\gamma-1/2} \, \int_0^\xi \exp(-\sigma)\,\sigma^{\gamma-1/2}\,d\sigma \\
& \simeq \xi^{-\gamma-1/2} \, \min\big(\xi^{\gamma+1/2},1\big) \simeq
\frac1{(1+\xi)^{\gamma+1/2}}.
\end{align*}
\end{proof}

\section{{Proof of Theorem \ref{thm:main}}} \label{sec:proof}

In this proof we may assume that $T=1$, since in compact time intervals separated from $0$ the short-time and long-time
bounds coincide, and the long-time bounds \eqref{1.XX} are a simple consequence of the short-time bounds (see \cite[p.\,233]{NoSj}).
In what follows all relations $\simeq$ and $\gtrsim$ are uniform with respect to $\theta,\varphi \in [0,\pi]$ and $0 < t \le 1$.

The proof of Theorem \ref{thm:main} is based on the five parts of Theorem \ref{thm:red_ext},
which are defined in  terms of $\a$ and $\b$.
In \textbf{Part (i) of Theorem \ref{thm:red_ext}} one has $\a,\b \ge -1/2$. This case is  contained in the known part of
Theorem \ref{thm:main} and will not be considered here.

\subsection*{{STEP 1: Part (ii) of Theorem \ref{thm:red_ext}}} \;

\medskip
\noindent \underline{Here  $-1 < \b < -1/2 \le \a$.}
\medskip

By Theorem \ref{thm:red_ext} (ii)
\begin{equation}\label{caseii}
G_t^{\a,\b}(\cos\theta,\cos\varphi) \simeq I_1 + I_2,
\end{equation}
where
\begin{align*}
I_1 & = - (\pi-\theta)(\pi-\varphi) \iint G_{t/4}^{\a+\b+3/2}\Big(u\sin\frac{\theta}2\sin\frac{\varphi}2
	+ v\cos\frac{\theta}2\cos\frac{\varphi}2 ,1\Big) \, d\Pi_{\a}(u)\, \Pi_{\b}(v)\, dv; \\
I_2 & = \iint G_{t/4}^{\a+\b+1/2}\Big(u\sin\frac{\theta}2\sin\frac{\varphi}2 + v\cos\frac{\theta}2\cos\frac{\varphi}2 ,1\Big)
	\, d\Pi_{\a}(u)\, d\Pi_{-1/2}(v).
\end{align*}
Observe that $I_1$ and $I_2$ are positive, in the case of $I_1$ since $G_t^{\a+\b+3/2}(\cdot,1)$ is increasing and
$\Pi_{\b}$ is odd and decreasing. We will estimate $I_1$ and $I_2$ separately.

For the simpler $I_2$ we have (the relations will be justified in a moment)
\begin{align*}
I_2 & \simeq \iint_{[0,1]^2} G_{t/4}^{\a+\b+1/2}\Big(u\sin\frac{\theta}2\sin\frac{\varphi}2
	+ v\cos\frac{\theta}2\cos\frac{\varphi}2 ,1\Big) \, d\Pi_{\a}(u)\, d\Pi_{-1/2}(v) \\
& \simeq t^{-\a-\b-3/2} \int_{[0,1]} \exp\bigg( -\frac{F(u,1)}t\bigg)\, d\Pi_{\a}(u).
\end{align*}
Here we first restrict the set of integration by means of the monotonicity of
$G_{t/4}^{\a+\b+1/2}(\cdot,1)$ (see Lemma \ref{lem:diff}) and the symmetry of the measures involved.
The second relation follows from plugging in the bound of Lemma \ref{lem:jac_low}.
Applying Lemma \ref{lem:exp} with $v=1$, we conclude
\begin{equation*}
I_2\gtrsim t^{-\a-\b-3/2} 
	\exp\bigg(-\frac{ F(1,1)}t\bigg) \int_{[0,1]} \exp\bigg(-c_1\,\frac{(1-u) \theta\varphi}t\bigg)\, d\Pi_{\a}(u),
\end{equation*}
and this inequality can be reversed if we replace $c_1$ by $c_2$. But for any fixed $c>0$, Lemma \ref{ssigma} implies
\begin{equation*}
   \int_{[0,1]} \exp\bigg(-c\,\frac{(1-u) \theta\varphi}t\bigg)\, d\Pi_{\a}(u)
   \simeq \frac1{(1+\theta\varphi/t)^{\a+1/2}} = t^{\a+1/2} (t+\theta\varphi)^{-\a-1/2}.
\end{equation*}
It follows that
\begin{equation} \label{I_2}
  I_2 \simeq (t+\theta\varphi)^{-\a-1/2}\, t^{-\b-1} \exp\bigg( -\frac{(\theta-\varphi)^2}{4t}\bigg).
\end{equation}

To treat $I_1$, we recall that $\Pi_{\b}(v)$ is an odd function which is negative for $v > 0$. Thus we have
\begin{multline*}
I_1 = (\pi-\theta)(\pi-\varphi) \int_{[-1,1]}\int_{[0,1]} \Big[
	G_{t/4}^{\a+\b+3/2}\Big(u\sin\frac{\theta}2\sin\frac{\varphi}2 + v\cos\frac{\theta}2\cos\frac{\varphi}2 ,1\Big) \\
 - G_{t/4}^{\a+\b+3/2}\Big(u\sin\frac{\theta}2\sin\frac{\varphi}2 - v\cos\frac{\theta}2\cos\frac{\varphi}2,1\Big)\Big]
		\, |\Pi_{\b}(v)|\, dv \, d\Pi_{\a}(u).
\end{multline*}
We write the difference here as the integral from $-v$ to $v$ of the derivative and use Lemma~\ref{lem:diff}, to get
\begin{align*}
  I_1  & \simeq
 (\pi-\theta)^2(\pi-\varphi)^2 \\ & \qquad \times \int_{[-1,1]}\int_{[0,1]}\int_{-v}^v
 G_{t/4}^{\a+\b+5/2}\Big(u\sin\frac{\theta}2\sin\frac{\varphi}2 +  v'\cos\frac{\theta}2\cos\frac{\varphi}2 ,1\Big)
 \,dv' \,|\Pi_{\b}(v)|\, dv \, d\Pi_{\a}(u).
\end{align*}
Since the integrand here is increasing in $u$ and in $v'$, the symmetries of the measures  $d\Pi_{\a}$ and $dv'$
show that one can restrict the integration in these two variables to $[0,1]$ and $[0,v]$, respectively.
Then we use Fubini's theorem and integrate first in $v$,
which produces a  factor $\int_{v'}^1|\Pi_{\b}(v)|\, dv \simeq (1-v')^{\b+3/2}$. As a result,
\begin{equation*}
I_1 \simeq
 (\pi-\theta)^2(\pi-\varphi)^2 \int_{[0,1]}\int_{[0,1]}
 G_{t/4}^{\a+\b+5/2}\Big(u\sin\frac{\theta}2\sin\frac{\varphi}2 + v'\cos\frac{\theta}2\cos\frac{\varphi}2 ,1\Big)
\, (1-v')^{\b+3/2}  \,dv'  \, d\Pi_{\a}(u).
\end{equation*}

Now Lemma  \ref{lem:jac_low}
 tells us that
\begin{equation*}
I_1 \simeq
 (\pi-\theta)^2(\pi-\varphi)^2\, t^{-\a-\b-7/2} \int_{[0,1]}\int_{[0,1]} \exp\bigg( -\frac{F(u,v')}t\bigg)
\, (1-v')^{\b+3/2}  \,dv'  \, d\Pi_{\a}(u).
\end{equation*}
From Lemma \ref{lem:exp} we see that this expression can be estimated from above and below by constant times
\begin{align*}
& (\pi-\theta)^2(\pi-\varphi)^2\, t^{-\a-\b-7/2} \exp\bigg(-\frac{F(1,1)}t\bigg) \\
& \;\; \times
  \int_{[0,1]} \exp\bigg(-c\,\frac{(1-u) \theta\varphi}t\bigg) \, d\Pi_{\a}(u)
\int_{[0,1]} \exp\bigg(-c\,\frac{ (1-v')(\pi - \theta)(\pi - \varphi)}t\bigg) \, (1-v')^{\b+3/2} \, dv',
\end{align*}
though with different   values of $c>0$ in the upper and lower estimates. In the last integral here,
we can replace $ (1-v')^{\b+3/2}  \,dv'$ by the comparable measure $d\Pi_{\b+2}(v')$. Then
Lemma \ref{ssigma} implies that the product of the two integrals is of order of magnitude
\begin{equation*}
   t^{\a+1/2} (t+\theta\varphi)^{-\a-1/2}\, t^{\b+5/2}\big[t+(\pi-\theta)(\pi-\varphi)\big]^{-\b-5/2}.
\end{equation*}

Summing up, we have
\begin{equation*} 
I_1 \simeq (t+\theta\varphi)^{-\a-1/2} (\pi-\theta)^2(\pi-\varphi)^2\, \,\big[t+(\pi-\theta)(\pi-\varphi)\big]^{-\b-5/2}
   \, t^{-1/2}  \,  \exp\bigg( -\frac{(\theta-\varphi)^2}{4t}\bigg).
\end{equation*}
Adding this and \eqref{I_2}, we obtain from \eqref{caseii}
\begin{equation} \label{iires}
G_t^{\a,\b}(\cos\theta,\cos\varphi)  \simeq
   (t+\theta\varphi)^{-\a-1/2} \,\big[t+(\pi-\theta)(\pi-\varphi)\big]^{-\b-1/2}
   \, t^{-1/2}  \,  \exp\bigg( -\frac{(\theta-\varphi)^2}{4t}\bigg),
\end{equation}
as desired.

\subsection*{STEP 2: {Part (iii) of Theorem \ref{thm:red_ext}}} \;

\medskip
\noindent \underline{Here $-1 < \a < -1/2 \le \b$.}
\medskip

This part is completely analogous to Part (ii).
We need only swap $\a$ with $\b$,\hskip4pt $u$ with $v$ and $\theta,\;\varphi$ with  $\pi-\theta,\;\pi-\varphi$.
The result will coincide with \eqref{iires}.

\subsection*{STEP 3: {The first three integrals in Parts  (iv) and (v)  of Theorem \ref{thm:red_ext}}}\;

\medskip
\noindent \underline{Here $-1 <\a,\b < -1/2$.}
\medskip

Neglecting irrelevant positive factors, these three integrals in Part (iv) coincide with those in  Part (v), and they are
\begin{align*}
J_1 & = \,\sin\theta\sin\varphi
\, \iint	G_{t/4}^{\a+\b+5/2}\Big(u \sin\frac{\theta}2\sin\frac{\varphi}2 +
	v\cos\frac{\theta}2\cos\frac{\varphi}2,1\Big)
 \, \Pi_{\a}(u)\, du \, \Pi_{\b}(v)\, dv;\\
J_2 & = \,-\sin\frac{\theta}2\sin\frac{\varphi}2
\,\iint	G_{t/4}^{\a+\b+3/2}\Big(u \sin\frac{\theta}2\sin\frac{\varphi}2 +
	v\cos\frac{\theta}2\cos\frac{\varphi}2,1\Big)
	\, \Pi_{\a}(u)\, du \, d\Pi_{-1/2}(v); \\
J_3 & = \,-\cos\frac{\theta}2\cos\frac{\varphi}2
\,\iint	G_{t/4}^{\a+\b+3/2}\Big(u \sin\frac{\theta}2\sin\frac{\varphi}2 +
	v\cos\frac{\theta}2\cos\frac{\varphi}2,1\Big)
	\, d\Pi_{-1/2}(u) \, \Pi_{\b}(v)\, dv.
\end{align*}
Each of these three quantities will be seen to be positive in the estimates that follow.

We start with $J_1$.
Imitating the first step of the treatment of $I_1$ above, now for both the variables $u$ and $v$, we have
\begin{equation} \label{j1}
J_1 = \,\sin\theta\sin\varphi
\, \iint_{[0,1]^2}  \Delta(u,v)
\, |\Pi_{\a}(u)|\, du \, |\Pi_{\b}(v)|\, dv,
\end{equation}
where $\Delta(u,v)$ is defined as
\begin{align*}
 & 	G_{t/4}^{\a+\b+5/2}\Big(u \sin\frac{\theta}2\sin\frac{\varphi}2 +
	v\cos\frac{\theta}2\cos\frac{\varphi}2,1\Big)
-G_{t/4}^{\a+\b+5/2}\Big(-u \sin\frac{\theta}2\sin\frac{\varphi}2 +
	v\cos\frac{\theta}2\cos\frac{\varphi}2,1\Big)   \\
 &  -\,G_{t/4}^{\a+\b+5/2}\Big(u \sin\frac{\theta}2\sin\frac{\varphi}2
	-v\cos\frac{\theta}2\cos\frac{\varphi}2,1\Big)
+G_{t/4}^{\a+\b+5/2}\Big(-u \sin\frac{\theta}2\sin\frac{\varphi}2 -
	v\cos\frac{\theta}2\cos\frac{\varphi}2,1\Big).
\end{align*}
This second difference is also given by
\begin{equation}\label{Delta}
 \Delta(u,v) =	\iint_{|u'|<u,\,|v'|<v}
\partial_{u'}\partial_{v'}	\Big[G_{t/4}^{\a+\b+5/2}\Big(u' \sin\frac{\theta}2\sin\frac{\varphi}2 +
	v'\cos\frac{\theta}2\cos\frac{\varphi}2,1\Big)\Big] \,du'\,dv'.
\end{equation}
Applying Lemma \ref{lem:diff} twice, we have
\begin{align*}
&\partial_{u'}\partial_{v'} \Big[	G_{t/4}^{\a+\b+5/2}\Big(u' \sin\frac{\theta}2\sin\frac{\varphi}2 +
	v'\cos\frac{\theta}2\cos\frac{\varphi}2,1\Big)\Big]\\
&\quad \simeq \sin\theta\sin\varphi\: G_{t/4}^{\a+\b+9/2}
\Big(u' \sin\frac{\theta}2\sin\frac{\varphi}2 +
	v'\cos\frac{\theta}2\cos\frac{\varphi}2,1\Big).
\end{align*}
By combining this with \eqref{j1} and \eqref{Delta}, we obtain
\begin{multline*}
J_1  \simeq  \,\sin^2\theta\,\sin^2\varphi \iint_{[0,1]^2}
\iint_{|u'|<u,\,|v'|<v}  \\ G_{t/4}^{\a+\b+9/2}\Big(u' \sin\frac{\theta}2\sin\frac{\varphi}2 +
	v'\cos\frac{\theta}2\cos\frac{\varphi}2,1\Big)\,  du'\,dv'
  \, |\Pi_{\a}(u)|\, du \, |\Pi_{\b}(v)|\, dv.
\end{multline*}

As in the case of $I_1$ in Step 1, we can restrict the inner integration here to $0<u'<u$ and $0<v'<v$,
then apply Lemma \ref{lem:jac_low}, and finally swap the order of integration. The result will be
\begin{equation*}
 J_1 \simeq \sin^2\theta\,\sin^2\varphi\;t^{-\a-\b-11/2}\iint_{[0,1]^2}
 \exp\bigg(-\frac{F(u',v')}{t}\bigg)\, \int_{u'}^{1}|\Pi_{\a}(u)|\, du\, \int_{v'}^{1}|\Pi_{\b}(v)|\, dv\, du'\, dv'.
\end{equation*}
From Lemma \ref{lem:mes} we see that the two inner integrals here are of order of magnitude
$(1-u')^{\a+3/2}$ and $(1-v')^{\b+3/2}$, respectively. The same lemma allows us to replace
$(1-u')^{\a+3/2}\,du'$ and  $(1-v')^{\b+3/2}\,dv'$ by $d\Pi_{\a+2}(u')$ and  $d\Pi_{\b+2}(v')$.
Then Lemma \ref{lem:exp} implies that $J_1$ is bounded above and below by expressions of type
\begin{align*}
 &\sin^2\theta \,\sin^2\varphi\;t^{-\a-\b-11/2} \exp\bigg(-\frac{F(1,1)}{t}\bigg)\, \\
&\quad \times \int_{[0,1]}
\,\exp\bigg(-c\,\frac{(1-u') \theta\varphi}t\bigg)\, d\Pi_{\a+2}(u') \;  \int_{[0,1]}
\exp\bigg(-c\,\frac{ (1-v')(\pi - \theta)(\pi - \varphi)}t\bigg)\, d\Pi_{\b+2}(v'),
\end{align*}
with constants $c>0$. Because of Lemma \ref{ssigma},  the integrals in this expression
are of orders of magnitude $(1+\theta\varphi/t)^{-\a-5/2}$ and
$[1+(\pi - \theta)(\pi - \varphi)/t]^{-\b-5/2}$, respectively. Thus
\begin{align}\label{J1res}
J_1 & \simeq \theta^2\varphi^2\,(t+\theta\varphi)^{-\a-5/2}\,(\pi - \theta)^2(\pi - \varphi)^2\,
  \big[t+(\pi - \theta)(\pi - \varphi)\big]^{-\b-5/2} \\
& \qquad \times t^{-1/2}\,\exp\bigg(-\frac{(\theta-\varphi)^2}{4t}\bigg).  \notag
\end{align}

Next we consider $J_2$. As in the argument for $I_1$ above but with $u$ and $v$ swapped, we write
\begin{align*}
J_2 &  =\,
   \sin\frac{\theta}2\sin\frac{\varphi}2
\,\int_{v\in [-1,1]}\int_{u\in[0,1]}	\Big[G_{t/4}^{\a+\b+3/2}\Big(u \sin\frac{\theta}2\sin\frac{\varphi}2 +
	v\cos\frac{\theta}2\cos\frac{\varphi}2,1\Big)  \\
&  \qquad  \qquad \qquad - G_{t/4}^{\a+\b+3/2}\Big(-u \sin\frac{\theta}2\sin\frac{\varphi}2 +
	v\cos\frac{\theta}2\cos\frac{\varphi}2,1\Big)\Big]
	\, |\Pi_{\a}(u)|\, du \, d\Pi_{-1/2}(v)\\
& \simeq \, \sin^2\frac{\theta}2\,\sin^2\frac{\varphi}2
\,\int_{v\in [-1,1]}\int_{u\in[0,1]} \int_{|u'|<u} \\
&  \qquad  \qquad \qquad  G_{t/4}^{\a+\b+5/2}\Big(u' \sin\frac{\theta}2\sin\frac{\varphi}2 +
	v\cos\frac{\theta}2\cos\frac{\varphi}2,1\Big)\,du'\, |\Pi_{\a}(u)|\, du \, d\Pi_{-1/2}(v),
\end{align*}
where we also wrote the difference as an integral of the derivative and used Lemma \ref{lem:diff}.
Since $G_{t/4}^{\a+\b+5/2}(\cdot,1)$ is increasing and the measures are symmetric,
the integrations in $u'$ and $v$ can be restricted to $0<u'<u$ and $v=1$, respectively.
As in the case of $J_1$ above, we switch the order of integration in
$u'$ and $u$, and then replace $\int_{u'}^{1} |\Pi_{\a}(u)|\, du\,du'$ by $d\Pi_{\a+2}(u')$.
The result will be
\begin{equation*}
  J_2  \simeq \,
  \theta^2\,\varphi^2
  \,\int_{[0,1]} G_{t/4}^{\a+\b+5/2}\Big(u' \sin\frac{\theta}2\sin\frac{\varphi}2 +
	\cos\frac{\theta}2\cos\frac{\varphi}2,1\Big)\,d\Pi_{\a+2}(u').
\end{equation*}
Apply now first Lemma \ref{lem:jac_low} and then Lemma \ref{lem:exp}, to conclude that
\begin{equation*}
J_2  \simeq \,
  \theta^2\,\varphi^2\; t^{-\a-\b-7/2} \, \exp\bigg(-\frac{F(1,1)}{t}\bigg)\,
  \int_{[0,1]} \,\exp\bigg(-c\,\frac{(1-u') \theta\varphi}t\bigg) \, d\Pi_{\a+2}(u').
\end{equation*}
Using finally Lemma \ref{ssigma}, we get
\begin{equation}  \label{J2res}
J_2  \simeq \, \theta^2\,\varphi^2 \; (t+\theta\varphi)^{-\a-5/2}\, 
	t^{-\b-1} \, \exp\bigg(-\frac{(\theta-\varphi)^2}{4t}\bigg). 
\end{equation}

The integral $J_3$ is obtained from $J_2$ by the same swapping that led from Part (ii) to Part (iii) above. It follows that
\begin{equation}  \label{J3res}
  J_3  \simeq \,(\pi - \theta)^2(\pi - \varphi)^2\, t^{-\a-1} \,
  \big[t+(\pi - \theta)(\pi - \varphi)\big]^{-\b-5/2}\, \exp\bigg(-\frac{(\theta-\varphi)^2}{4t}\bigg).
\end{equation}

\subsection*{STEP 4: {The fourth integral in Part  (iv) of Theorem \ref{thm:red_ext}}} \;

\medskip
\noindent \underline{Here $-1 <\a,\b < -1/2$ and $\a+\b > -3/2$.}
\medskip

This integral is
\begin{align*}
 & \iint G_{t/4}^{\a+\b+1/2}\Big(u \sin\frac{\theta}2\sin\frac{\varphi}2 +
	v \cos\frac{\theta}2\cos\frac{\varphi}2,1\Big)\,d\Pi_{-1/2}(u)\,d\Pi_{-1/2}(v)\\
 &\quad \simeq \, G_{t/4}^{\a+\b+1/2}\Big( \sin\frac{\theta}2\sin\frac{\varphi}2 +
	 \cos\frac{\theta}2\cos\frac{\varphi}2,1\Big),
\end{align*}
since $G_{t/4}^{\a+\b+1/2}(\cdot,1)$ is increasing. Lemma \ref{lem:jac_low} tells us that this is of order of magnitude
\begin{equation} \label{J4res}
  t^{-\a-\b-3/2}\, \exp\bigg(-\frac{(\theta-\varphi)^2}{4t}\bigg).
\end{equation}

We can now finish the argument for Part (iv) of  Theorem \ref{thm:red_ext}.
The quantities in \eqref{J4res} and those in \eqref{J1res}, \eqref{J2res} and \eqref{J3res} are all dominated by
$Z_t^{\a,\b}(\theta,\varphi)$.
Moreover, depending on the relative sizes of  $\theta\varphi$, $(\pi - \theta)(\pi - \varphi)$ and $t$,
one of these quantities is in each case of the same order of magnitude as $Z_t^{\a,\b}(\theta,\varphi)$.  

\subsection*{STEP 5: {The fourth integral in Part (v) of Theorem \ref{thm:red_ext}}} \;

\medskip
\noindent \underline{Here $-1 <\a,\b < -1/2$ and $\a+\b \le -3/2$.}
\medskip

The integral is
\begin{equation} \label{term4}
\iint  H_{t/4}^{\lambda}\Big(u \sin\frac{\theta}2\sin\frac{\varphi}2 + v\cos\frac{\theta}2\cos\frac{\varphi}2 \Big)
	\, d\Pi_{-1/2}(u)\, d\Pi_{-1/2}(v),
\end{equation}
where $\lambda :=\a+\b+1/2 \in (-3/2,-1]$.
We can now use Lemma \ref{Htl}, together with the facts that $H_{t}^{\lambda}$ is even in $[-1,1]$
and increasing in $[0,1]$, to estimate the integral in \eqref{term4}. This integral equals one-half of
\begin{equation*}
  H_{t/4}^{\lambda}\bigg(\cos \frac{\theta -\varphi}2\bigg) +  H_{t/4}^{\lambda}\bigg(\cos \frac{\theta +\varphi}2\bigg) \simeq
  H_{t/4}^{\lambda}\bigg(\cos \frac{\theta -\varphi}2\bigg) \simeq t^{-\lambda -1}\, \exp\bigg( -\frac{(\theta-\varphi)^2}{4t} \bigg).
\end{equation*}

By means of these estimates together with those in \eqref{J1res}, \eqref{J2res} and \eqref{J3res},
we can now finish the proof for Part (v) of Theorem  \ref{thm:red_ext}, in the same way as we ended that of
Part (iv) in Step 4.

Theorem \ref{thm:main} is completely proved.
\qed


\end{document}